\documentclass[12pt]{article}
\usepackage{graphicx}
\usepackage{amsmath,amsthm,amssymb,enumerate}
\usepackage{color}
\usepackage[utf8]{inputenc}
\catcode`\@=11 \@addtoreset{equation}{section}

\catcode`\@=12

\newtheorem{Theorem}{Theorem}[section]
\newtheorem{Proposition}{Proposition}[section]
\newtheorem{Lemma}{Lemma}[section]
\newtheorem{Corollary}{Corollary}[section]
\newtheorem{Remark}{Remark}[section]
\newtheorem{Definition}{Definition}[section]

\newcommand{\bTheorem}[1]{
\begin{Theorem} \label{T#1} }
\newcommand{\eT}{\end{Theorem}}

\newcommand{\bProposition}[1]{
\begin{Proposition} \label{P#1}}
\newcommand{\eP}{\end{Proposition}}

\newcommand{\bLemma}[1]{
\begin{Lemma} \label{L#1} }
\newcommand{\eL}{\end{Lemma}}

\newcommand{\bCorollary}[1]{
\begin{Corollary} \label{C#1} }
\newcommand{\eC}{\end{Corollary}}

\newcommand{\bRemark}[1]{
\begin{Remark} \label{R#1} }
\newcommand{\eR}{\end{Remark}}

\newcommand{\bDefinition}[1]{
\begin{Definition} \label{D#1} }
\newcommand{\eD}{\end{Definition}}

\newcommand{\bFormula}[1]{ \begin{equation} \label{#1} }
\newcommand{\eF}{ \end{equation} }

\newcommand{\Ov}[1]{\overline{#1}}

\newcommand{\vr}{\varrho}

\newcommand{\vu}{\vc{u}}
\newcommand{\vc}[1]{{\bf #1}}

\newcommand{\Div}{{\rm div}_x}
\newcommand{\Grad}{\nabla_x}

\newcommand{\tn}[1]{\mathbb{ #1 }}
\newcommand{\dx}{{\rm d} {x}}
\newcommand{\dt}{{\rm d} t }

\newcommand{\dxdt}{\dx \ \dt}
\newcommand{\intO}[1]{\int_{\Omega} #1 \ \dx}

\newcommand{\bProof}{{\bf Proof: }}

\definecolor{Cgrey}{rgb}{0.85,0.85,0.85}
\definecolor{Cblue}{rgb}{0.50,0.85,0.85}
\definecolor{Cred}{rgb}{1,0,0}
\definecolor{fancy}{rgb}{0.10,0.85,0.10}

\newcommand\Cbox[2]{%
    \newbox\contentbox%
    \newbox\bkgdbox%
    \setbox\contentbox\hbox to \hsize{%
        \vtop{
            \kern\columnsep
            \hbox to \hsize{%
                \kern\columnsep%
                \advance\hsize by -2\columnsep%
                \setlength{\textwidth}{\hsize}%
                \vbox{
                    \parskip=\baselineskip
                    \parindent=0bp
                    #2
                }%
                \kern\columnsep%
            }%
            \kern\columnsep%
        }%
    }%
    \setbox\bkgdbox\vbox{
        \color{#1}
        \hrule width  \wd\contentbox %
               height \ht\contentbox %
               depth  \dp\contentbox
        \color{black}
    }%
    \wd\bkgdbox=0bp%
    \vbox{\hbox to \hsize{\box\bkgdbox\box\contentbox}}%
    \vskip\baselineskip%
}

\date{}

\begin{document}


\title{Dissipative measure-valued solutions to the compressible Navier-Stokes system}

\author{Eduard Feireisl \thanks{The research of E.F.\ leading to these results has received funding from the European Research Council under the European Union's Seventh Framework Programme (FP7/2007-2013)/ ERC Grant Agreement 320078. The Institute
of Mathematics of the Academy of Sciences of the Czech Republic is supported by RVO:67985840. }\and Piotr Gwiazda\thanks{ The research of A.\'S-G  and P.G. has received funding from the  National Science Centre, Poland, 2014/13/B/ST1/03094} \and Agnieszka \' Swierczewska-Gwiazda \and Emil Wiedemann}

\maketitle

\centerline{Institute of Mathematics of the Academy of Sciences of the Czech Republic}

\centerline{\v Zitn\' a 25, CZ-115 67 Praha 1, Czech Republic}

\bigskip

\centerline{Institute of Applied Mathematics and Mechanics, University of Warsaw}

\centerline{Banacha 2, 02-097 Warszawa, Poland}

\bigskip

\centerline{Institute of Applied Mathematics and Mechanics, University of Warsaw}

\centerline{Banacha 2, 02-097 Warszawa, Poland}

\bigskip

\centerline{Hausdorff Center for Mathematics and Mathematical Institute, University of Bonn}

\centerline{Endenicher Allee 60, 53115 Bonn, Germany}

\bigskip






\maketitle

\bigskip





\begin{abstract}

We introduce a new concept of dissipative measure-valued solution to the compressible Navier-Stokes system satisfying, in addition, a relevant form of the total energy balance. Then we show that a dissipative measure-valued and a standard smooth classical solution originating from the same initial data coincide (weak-strong uniqueness principle) as long as the latter exists. Such a result facilitates considerably the proof of convergence of solutions to various approximations including certain numerical schemes that are known to generate a measure-valued solution. As a byproduct we show that any measure-valued solution with
bounded density component that starts from smooth initial data is necessarily a classical one.

\end{abstract}




\section{Introduction}
\label{I}

The concept of measure-valued solution to partial differential equations was introduced by DiPerna~\cite{DiP} in the context of conservation laws. He used Young measures in order to conveniently pass to the artificial viscosity limit and in some situations (e.g.\ the scalar case) proved a posteriori that the measure-valued solution is atomic, i.e.\ it is in fact a solution in the sense of distributions.

For general systems of conservation laws there is no hope to obtain (entropy) solutions in the distributional sense and therefore there seems to be no alternative to the use of measure-valued solutions or related concepts. In the realm of inviscid fluid dynamics, the existence of measure-valued solutions has been established for a variety of models~\cite{DiMa, Neustup, Gwi}.

Measure-valued solutions to problems involving \emph{viscous} fluids were introduced in the early nineties in \cite{MNRR} and may seem obsolete nowadays in the light of the theory proposed by P.-L. Lions \cite{LI4} and extended by Feireisl et al.\ \cite{FNP} in the framework of weak solutions for the barotropic Navier-Stokes system
\begin{equation}\label{NS}
\begin{aligned}
\partial_t \vr + \Div (\vr \vu) &=0 \\
 \partial_t (\vr \vu) + \Div (\vr \vu \otimes \vu) + \Grad p(\vr) &= \Div \mathbb{S} (\Grad \vu),
\end{aligned}
\end{equation}
where $\vr$ is the density, $\vu$ the velocity, $p$ the given pressure function, and $\mathbb{S}$ the Newtonian viscous stress.

The reason we consider measure-valued solutions nevertheless is twofold: First, the results of this paper pertain to any adiabatic exponent greater than one, whereas the known existence theory for weak solutions requires $\gamma>3/2$; second, there remains a vast class of approximate problems including systems with higher order viscosities and solutions to certain numerical schemes for which it is rather easy to show that they generate a measure-valued solution whereas convergence to a weak solution is either not
known or difficult to prove. This motivates the present study, where we introduce a new concept of (dissipative) measure valued solution to the system~\eqref{NS}.

The main novelty is that we have to deal with nonlinearities both in the velocity and its derivative, since we need to make sense of the energy inequality
\[
\partial_t \int_{\Omega} \left[ \frac{1}{2} \vr |\vu|^2 + P(\vr) \right]  +
\int_\Omega \left[ \mathbb{S} (\Grad \vu) : \Grad \vu\right]  \leq 0
\]
in the measure-valued framework. Indeed, Neustupa~\cite{Neustup} considered measure-valued solutions of~\eqref{NS}, but his theory does not involve the energy. Young measures do not seem suitable to describe the limit distributions of a map and its gradient simultaneously, as it is unclear how the information that one component of the measure is in some sense the gradient of the other component is encoded in the measure. We solve this issue by introducing a ``dissipation defect" (see Definition \ref{DD1}), which encodes all conceivable concentration effects in the density and the velocity, and concentration \emph{and} oscillation effects in the gradient of the velocity. It then turns out that postulating a Poincar\'e-type inequality (see \eqref{KoPo}), which is satisfied by any measure generated by a reasonable approximating sequence of solutions for~\eqref{NS}, already suffices to ensure weak-strong uniqueness. As a side effect, we thus avoid the notationally somehow heavy framework
  of Alibert and Bouchitt\'e~\cite{AlBo} and give the most extensive definition of dissipative measure-valued solution that still complies with weak-strong uniqueness (cf.\ also the discussion in Section~\ref{dissipative}).

Indeed, the proof of weak-strong uniqueness for our dissipative measure-valued solutions is the main point of this paper (Theorem~\ref{TT1}). Weak-strong uniqueness means that classical solutions are stable within the class of dissipative measure-valued solutions. For the incompressible Navier-Stokes equations, a weak-strong uniqueness principle was shown for the first time in the classical works of Prodi~\cite{Pro} and Serrin~\cite{Ser}. Surprisingly, even in the measure-valued setting, weak-strong uniqueness results have been proved: For the incompressible Euler equations and bounded solutions of conservation laws this was done in~\cite{BrDLSz}, and for the compressible Euler system and a related model for granular flow in~\cite{GwSwWi}. In the context of elastodynamics, dissipative measure-valued solutions and their weak-strong uniqueness property were studied in~\cite{DeStTz}. Here, we give the first instance of weak-strong uniqueness for measure-valued solutions of a viscous fluid model.

We also identify a large class of problems generating dissipative measure-valued solutions including the pressure-density equations of state that are still beyond the reach of the current theory of weak solutions. We make a similar observation for certain numerical schemes, thus adopting the viewpoint of Fjordholm et al.\ \cite{FjKaMiTa}, who argue (in the context of hyperbolic systems of conservation laws) that dissipative measure-valued solutions are a more appropriate solution concept compared to weak entropy solutions, because the former are obtained as limits of common numerical approximations whereas the latter aren't.

As a further application of weak-strong unqiueness, we show (Theorem~\ref{TT2}) that every approximate sequence of solutions of~\eqref{NS} with uniformly bounded density converges to the unique smooth solution.

\section{Definition and existence of dissipative measure-valued solutions}
\subsection{Motivation: Brenner's model in fluid dynamics}
\label{BM}

To motivate our definition of measure-valued solution, we consider a model of a viscous compressible fluid proposed by Brenner \cite{BREN}, where the density
$\vr = \vr(t,x)$ and the velocity $\vu = \vu(t,x)$ satisfy
\begin{eqnarray}
\label{B1}
\partial_t \vr + \Div (\vr \vu) &=& K \Delta \vr \\
\label{B2}  \partial_t (\vr \vu) + \Div (\vr \vu \otimes \vu) + \Grad p(\vr) &=& \Div \mathbb{S} (\Grad \vu)+  K \Div (\vu \otimes \Grad \vr),
\end{eqnarray}
where $K > 0$ is a parameter, and $\mathbb{S}$ the standard Newtonian viscous stress
\begin{equation} \label{B3}
\mathbb{S}(\Grad \vu) = \mu \left( \Grad \vu + \Grad^t \vu - \frac{2}{3} \Div \vu \mathbb{I} \right) + \eta \Div \vu \mathbb{I},\ \mu > 0,\ \eta \geq 0.
\end{equation}
Note that $\mathbb{S}$ depends only on the symmetric part of its argument.
Problem (\ref{B1}--\ref{B3}) may be supplemented by relevant boundary conditions, here
\begin{equation} \label{B4}
\vu|_{\partial \Omega} = 0, \ \Grad \vr \cdot \vc{n}|_{\partial \Omega} = 0,
\end{equation}
where $\Omega \subset R^N$, $N=2,3$ is a regular bounded domain.

In addition, sufficiently smooth solutions of (\ref{B1}--\ref{B4}) obey the total energy balance:
\begin{equation} \label{B5}
\partial_t \intO{ \left[ \frac{1}{2} \vr |\vu|^2 + P(\vr) \right] } +
\intO{ \left[ \tn{S} (\Grad \vu) : \Grad \vu + K P''(\vr) |\Grad \vr|^2 \right]} = 0,
\end{equation}
where $P$ denotes the pressure potential,
\[
P(\vr) = \vr \int_1^\vr \frac{p(z)}{z^2} \ {\rm d}z.
\]

Leaving apart the physical relevance of Brenner's model, discussed and criticized in several studies (see, e.g., \"{O}ttinger, Struchtrup, and Liu \cite{OeStLi}),
we examine the limit of a family of solutions $\{ \vr^K, \vu^K \}_{K > 0}$ for $K \to 0$. Interestingly, system (\ref{B1}--\ref{B3}) is almost identical to the
approximate problem used in \cite{EF70} in the construction of weak solutions to the barotropic Navier-Stokes system, in particular, the existence of $\{ \vr^K, \vu^K \}_{k > 0}$
for a fairly general class of initial data may be established by the method detailed in \cite[Chapter 7]{EF70}. A more general model of a heat-conducting fluid
based on Brenner's ideas has been also analyzed in \cite{FeiVas}.

We suppose that the energy of the initial data is bounded
\[
\intO{ \left[ \frac{1}{2} \vr^K_0 |\vu^K_0|^2 + P(\vr^K_0) \right] } \leq c
\]
uniformly for $K \to 0$. In order to deduce uniform bounds, certain coercivity assumption must be imposed on the pressure term:
\begin{equation} \label{BB6}
p \in C[0, \infty) \cap C^2(0, \infty), \ p(0) = 0, \ p'(\vr) > 0 \ \mbox{for}\ \vr > 0, \ \liminf_{\vr \to \infty} p'(\vr) > 0,\
\liminf_{\vr \to \infty} \frac{P(\vr)}{p(\vr)} > 0.
\end{equation}
Seeing that $P''(\vr) = p'(\vr)/\vr$ we deduce from the energy balance (\ref{B5}) the following bounds
\begin{equation} \label{BB7}
\begin{split}
\sup_{\tau \in [0,T]} \intO{ P(\vr^K) (\tau, \cdot) } \leq c \ & \Rightarrow \sup_{\tau \in [0,T]} \intO{ \vr^K \log(\vr^K) (\tau, \cdot) } \leq c, \\
\sup_{\tau \in [0,T]} \intO{ \vr^K |\vu^K |^2 (\tau, \cdot) } & \leq c, \\
\int_0^T \intO{ \mathbb{S}(\Grad \vu^K) : \Grad \vu^K  } \leq c \ & \Rightarrow \ \mbox{(Korn inequality)}\ \int_0^T \intO{ | \Grad \vu^K |^2  } \leq c \\
& \Rightarrow \ \mbox{(Poincar\' e inequality)}\ \int_0^T \intO{ | \vu^K |^2  } \leq c,\\
K \int_0^T \intO{ \frac{p'(\vr^K)}{\vr^K} |\Grad \vr^K |^2 } & \leq c
\end{split}
\end{equation}
uniformly for $K \to 0$.

Now, system (\ref{B1}, \ref{B2}) can be written in the weak form
\begin{equation} \label{wB1}
\left[ \intO{ \vr^K \psi} \right]_{t = 0}^{t = \tau} =
\int_0^\tau \intO{ \Big[ \vr^K \partial_t \psi + \vr^K \vu^K \cdot \Grad \psi - K \Grad \vr^K \cdot \Grad \psi \Big] } \ \dt
\end{equation}
for any $\psi \in C^1([0,T] \times \Ov{\Omega})$,
\begin{equation} \label{wB2}
\begin{split}
&\left[ \intO{ \vr^K \vu^K \cdot \varphi} \right]_{t = 0}^{t = \tau}
= \int_0^\tau \intO{ \Big[ \vr^K \vu^K \cdot \partial_t \varphi + \vr^K (\vu^K \otimes \vu^K ) : \Grad \varphi + p(\vr^K) \Div \varphi \Big] } \ \dt\\
& - \int_0^\tau \intO{  \Big[ \mathbb{S}(\Grad \vu^K ) : \Grad \varphi
+ K (\vu^K \otimes \Grad \vr^K) : \Grad \varphi  \Big] } \ \dt
\end{split}
\end{equation}
for any $\varphi \in C^1([0,T] \times \Ov{\Omega})$, $\varphi|_{\partial \Omega} = 0$.

The first observation is that the $K$-dependent quantities vanish in the
asymptotic limit $K \to 0$ as long as (\ref{BB7}) holds. To see this, note that
\begin{equation}
\begin{split} \label{IB1}
K \int_0^\tau \intO{  \Grad \vr^K \cdot \Grad \psi } \ \dt &= \sqrt{K} \int_0^\tau \intO{ \sqrt{K} \frac{\Grad \vr^K}{\sqrt{\vr^K}} \cdot \sqrt{\vr^K} \Grad \psi } \ \dt,\\
K \int_0^\tau \intO{ (\vu^K \otimes \Grad \vr^K ) : \Grad \varphi } \ \dt &= \sqrt{K} \int_0^\tau \intO{ \left( \sqrt{\vr^K} \vu^K \otimes \sqrt{K} \frac{\Grad \vr^K}{\sqrt{\vr^K}}\right ) : \Grad \varphi };
\end{split}
\end{equation}
whence, by virtue of hypothesis (\ref{BB6}), these integrals are controlled by (\ref{BB7}) at least on the set where $\vr^K \geq 1$. In order to estimate $\Grad \vr^K$ on the set where $\vr^K$ is small, we multiply (\ref{B1}) on $b'(\vr^K)$ obtaining
\begin{equation} \label{RB1}
\partial_t b(\vr^K) + \Div (b(\vr^K) \vu^K ) + \left( b'(\vr^K) \vr^K - b(\vr^K) \right)\Div \vu^K = K \Div \left( b'(\vr^K) \Grad b(\vr^K) \right) - Kb''(\vr^K) |\Grad \vr^K|^2.
\end{equation}
Such a step can be rigorously justified for the solutions of Brenner's problem discussed in \cite{FeiVas} provided, for instance, $b \in C^\infty_c[0, \infty)$. Thus taking
$b$ such that $b(\vr) = \vr^2$  for $\vr \leq 1$, integrating  (\ref{RB1}) and using (\ref{BB7}) we deduce that
\[
K \int\int_{ \{ \vr^K \leq 1 \} } |\Grad \vr^K |^2 \ \dxdt \leq c \ \mbox{uniformly for}\ K \to 0,
\]
which provides the necessary bounds for the integrals in (\ref{IB1}) on the set where $\vr^K \leq 1$. Indeed using the fact that $b$ is bounded and
the bounds established in (\ref{BB7}) we deduce
\[
\left| K \int_0^T \intO{ b''(\vr^K) |\Grad \vr^K |^2 } \ \dt \right| \leq c.
\]
On the other hand, thanks to our choice of $b$,
\[
K \int\int_{ \{ \vr^K \leq 1 \} } |\Grad \vr^K |^2 \ \dxdt = K \int_0^T \intO{ b''(\vr^K) |\Grad \vr^K |^2 } \ \dt -
K \int\int_{ \{ \vr^K > 1 \} } b''(\vr^K) |\Grad \vr^K |^2 \ \dxdt,
\]
where the right-most integral is bounded in view of (\ref{BB7}), hypothesis (\ref{BB6}) and the fact that $b''(\vr^K)$ vanishes for large $\vr^K$.

\color{black}
Consequently, we may, at least formally, let $K \to 0$ in (\ref{wB1}),
in (\ref{wB2}) and also in (\ref{B5}) obtaining a measure-valued solution to the barotropic Navier-Stokes system:
\begin{eqnarray}
\label{I1}
\partial_t \vr + \Div (\vr \vu) &=& 0,  \\
\label{I2}  \partial_t (\vr \vu) + \Div (\vr \vu \otimes \vu) + \Grad p(\vr) &=& \Div \mathbb{S} (\Grad \vu), \\
\label{I3} \vu|_{\partial \Omega} &=& 0.
\end{eqnarray}
More specifically, as all integrands in (\ref{wB1}), (\ref{wB2}) admit uniform bounds at least in the Lebesgue norm $L^1$, it is convenient to use the well developed framework of
parametrized measures associated to the family of equi-integrable functions $\{ \vr^K, \vu^K \}_{K > 0}$ generating a Young measure
\[
\nu_{t,x} \in \mathcal{P} \left([0, \infty) \times R^N \right) \ \mbox{for a.a.}\ (t,x) \in (0,T) \times \Omega,
\]
cf. Pedregal \cite[Chapter 6, Theorem 6.2]{PED1}. We will systematically use the notation
\[
\Ov{ F(\vr, \vu) }(t,x) = \langle \nu_{t,x} ; F(s, \vc{v}) \rangle
\ \mbox{for the dummy variables}\ s \approx \vr, \ \vc{v} \approx \vu.
\]


Focusing on the energy balance (\ref{B5}) we first take advantage of the no-slip boundary conditions and rewrite the viscous dissipation term in a more
convenient form
\[
\intO{ \mathbb{S} (\Grad \vu^K) : \Grad \vu^K } = \intO{ \left[ \mu |\Grad \vu^K |^2 + \lambda |\Div \vu^K |^2 \right] },\
\lambda = \frac{\mu}{3} + \eta > 0.
\]
Now, we identify
\[
\left[ \frac{1}{2} \vr^K |\vu^K|^2 + P(\vr^K) \right] (\tau, \cdot) \in \mathcal{M} (\Ov{\Omega}) \ \mbox{bounded uniformly for}\ \tau \in [0,T];
\]
\[
\left[ \mu |\Grad \vu^K |^2 + \lambda |\Div \vu^K |^2 \right] \ \mbox{bounded in}\ \mathcal{M}^+ ([0,T] \times \Ov{\Omega});
\]
whence, passing to a subsequence as the case may be, we may assume that
\[
\begin{split}
\left[ \frac{1}{2} \vr^K |\vu^K|^2 + P(\vr^K) \right] (\tau, \cdot) & \to E \ \mbox{weakly-(*) in}\
L^\infty_{\rm weak}(0,T; \mathcal{M}(\Ov{\Omega})),\\
\left[ \mu |\Grad \vu^K |^2 + \lambda |\Div \vu^K |^2 \right] &\to \sigma \ \mbox{weakly-(*) in}\ \mathcal{M}^+ ([0,T] \times \Ov{\Omega}).
\end{split}
\]
Thus, introducing new (non-negative) measures
\[
E_\infty = E - \langle \nu_{t,x} ; \frac{1}{2} s |\vc{v}|^2 + P(s) \rangle \ \dx, \ \sigma_\infty = \sigma - \left[
\mu |\nabla \langle \nu_{t,x};\vc{v}\rangle|^2 + \lambda \left( {\rm tr}|\nabla \langle \nu_{t,x};\vc{v}\rangle|\right)^2 \right]
\ \dxdt,
\]
we may perform the limit $K \to 0$ in the energy balance (\ref{B5}) obtaining
\begin{equation} \label{mvEI}
\begin{split}
&\intO{ \Ov{ \left( \frac{1}{2} \vr |\vu|^2 + P(\vr) \right) }(\tau, \cdot) } + E_\infty(\tau)[\Ov{\Omega}] + \int_0^\tau \intO{{\mu |\Grad\Ov{\vu}|^2 + \lambda |\Div \Ov{\vu}|^2 } } \ \dt + \sigma_\infty [[0,\tau] \times \Ov{\Omega}] \\
&\leq \intO{ \Ov{ \left( \frac{1}{2} \vr_0 |\vu_0|^2 + P(\vr_0) \right) }  } + E_\infty(0)[\Ov{\Omega}]
\end{split}
\end{equation}
for a.a. $\tau \in (0,T)$.

Applying a similar treatment to (\ref{wB1}) we deduce
\begin{equation} \label{mvB1}
\left[ \intO{ \Ov{\vr} \psi} \right]_{t = 0}^{t = \tau} =
\int_0^\tau \intO{ \Big[ \Ov{\vr} \partial_t \psi + \Ov{\vr \vu} \cdot \Grad \psi \Big] } \ \dt
\end{equation}
for any $\psi \in C^1([0,T] \times \Ov{\Omega})$. Note that (\ref{mvB1}) holds for any $\tau$ as the family
$\{ \vr^K \}_{K > 0}$ is precompact in $C_{\rm weak}([0,T]; L^1(\Omega))$. Indeed precompactness follows from the uniform bound for $\{\rho^K\}$ in $L\log L$ in \eqref{BB7}.

Our final goal is to perform the limit $K \to 0$ in (\ref{wB2}). This is a bit more delicate as both the convective term $\vr^K \vu^K \otimes \vu^K$ and
$p(\vr^K)$ are bounded only in $L^\infty(L^1)$. We use the following result:

\begin{Lemma}\label{Lmeas}
Let $\{ \vc{Z}_n \}_{n = 1}^\infty$, $\vc{Z}_n : Q \to R^N$ be a sequence of equi-integrable functions generating a Young measure $\nu_y$, $y \in Q$, where
$Q \subset R^M$ is a bounded domain. Let
\[
G: R^N \to [0, \infty)
\]
be a continuous function such that
\[
\sup_{n \geq 0} \| G(\vc{Z}_n) \|_{L^1(Q)} < \infty,
\]
and let $F$ be continuous such that
\[
F : R^N \to R \ \ |F(\vc{Z})| \leq G(\vc{Z}) \ \mbox{for all}\ \vc{Z} \in R^N.
\]
Denote
\[
F_\infty = \tilde F - \left< \nu_y , F(Z) \right> {\rm d}y ,\ G_\infty = \tilde G - \left< \nu_y , G(Z) \right> {\rm d}y,
\]
where $\tilde F \in \mathcal{M}(\Ov{Q})$, $\tilde G \in \mathcal{M} (\Ov{Q})$ are the weak-star limits of $\{ F(\vc{Z}_n) \}_{n \geq 1}$, $\{ G(\vc{Z}_n) \}_{n \geq 1}$
in $\mathcal{M}(\Ov{Q})$.

Then
\[
| F_\infty | \leq G_\infty.
\]

\end{Lemma}

\bProof

Write

\[
<\tilde F, \phi > = \lim_{n \to \infty} \int_{ |\vc{Z}_n| \leq M} F(\vc{Z}_n) \phi \ {\rm d}y + \lim_{n \to \infty} \int_{ |\vc{Z}_n| > M} F(\vc{Z}_n) \phi \ {\rm d}y,
\]
\[
<\tilde G, \phi > = \lim_{n \to \infty} \int_{ |\vc{Z}_n| \leq M} G(\vc{Z}_n) \phi \ {\rm d}y + \lim_{n \to \infty} \int_{ |\vc{Z}_n| > M} G(\vc{Z}_n) \phi \ {\rm d}y.
\]

Applying Lebesgue theorem, we get
\[
\lim_{M \to \infty} \left( \lim_{n \to \infty} \int_{ |\vc{Z}_n| \leq M} F(\vc{Z}_n) \phi \ {\rm d}y \right) = \int_Q \left< \nu_y ; F(\vc{Z}) \right> \ {\rm d}y,
\]
\[
\lim_{M \to \infty} \left( \lim_{n \to \infty} \int_{ |\vc{Z}_n| \leq M} G(\vc{Z}_n) \phi \ \dx \right) = \int_Q \left< \nu_y ; G(\vc{Z}) \right> \ {\rm d}y.
\]

Consequently,
\[
\left< F_\infty ; \phi \right> = \lim_{M \to \infty} \left( \lim_{n \to \infty} \int_{ |\vc{Z}_n| > M} F(\vc{Z}_n) \phi \ {\rm d}y \right), \
\left< G_\infty ; \phi \right> = \lim_{M \to \infty} \left( \lim_{n \to \infty} \int_{ |\vc{Z}_n| > M} G(\vc{Z}_n) \phi \ {\rm d}y  \right).
\]
As $|F| \leq G$ the desired result follows.

\qed

Seeing that
\[
|\vr u_i u_j| \leq \vr |\vu|^2 \ \mbox{and, by virtue of hypothesis (\ref{BB6})},\ p(\vr) \leq a P(\vr) \ \mbox{for}\ \vr >> 1,
\]
we may let $K \to 0$ in (\ref{wB2}) to deduce
\begin{equation} \label{mvB2}
\begin{split}
&\left[ \intO{ \Ov{\vr \vu} \cdot \varphi} \right]_{t = 0}^{t = \tau}
= \int_0^\tau \intO{ \Big[ \Ov{\vr \vu} \cdot \partial_t \varphi + \Ov{ \vr (\vu \otimes \vu ) } : \Grad \varphi + \Ov{p(\vr)} \Div \varphi \Big] } \ \dt\\
& - \int_0^\tau \intO{  \Ov{ \mathbb{S}(\Grad \vu ) } : \Grad \varphi } \ \dt + \int_0^\tau \left< {r}^M ; \Grad \varphi \right> \ \dt
\end{split}
\end{equation}
for any $\varphi \in C^1([0,T] \times \Ov{\Omega})$, $\varphi|_{\partial \Omega} = 0$, where
\[
r^M = \left\{ r^M_{i,j} \right\}_{i,j=1}^N\, r^M_{i,j} \in L^\infty_{\rm weak} (0,T; \mathcal{M}(\Ov{\Omega})),\ |r^M_{i,j}(\tau)| \leq c E_\infty (\tau)
\ \mbox{for a.a.} \ \tau \in (0,T).
\]

\subsection{Dissipative measure-valued solutions to the Navier-Stokes system}\label{dissipative}

Motivated by the previous considerations, we introduce the concept of \emph{dissipative measure valued solution} to the barotropic Navier-Stokes system.
\begin{Definition} \label{DD1}
We say that a parameterized measure $\{ \nu_{t,x} \}_{(t,x) \in (0,T) \times \Omega }$,
\[
\nu \in L^{\infty}_{\rm weak}\left( (0,T) \times \Omega; \mathcal{P} \left([0,\infty) \times R^N \right) \right),\
\left< \nu_{t,x}; s \right> \equiv \vr,\ \left< \nu_{t,x}; \vc{v} \right> \equiv \vu
\]
is a dissipative measure-valued solution of the Navier-Stokes system (\ref{I1} -- \ref{I3}) in
$(0,T) \times \Omega$, with the initial conditions $\nu_0$ and dissipation defect $\mathcal{D}$,
\[
\mathcal{D} \in L^\infty(0,T), \ \mathcal{D} \geq 0,
\]
if the following holds.
\begin{itemize}
\item {\bf Equation of continuity.}
There exists a measure $r^C\in L^1([0,T];\mathcal{M}(\Ov{\Omega}))$ and $\chi\in L^1(0,T)$ such that for a.a.\ $\tau\in(0,T)$ and every $\psi \in C^1([0,T] \times \Ov{\Omega})$,
\begin{equation}
\left| \langle r^C (\tau) ; \Grad\psi \rangle \right| \leq \chi(\tau)  \mathcal{D} (\tau) \| \psi \|_{C^1(\Ov{\Omega})}
\end{equation}
and
\begin{equation} \label{dmvB1}
\begin{split}
\intO{ \langle \nu_{\tau,x}; s \rangle \psi (\tau, \cdot) } &-  \intO{ \langle \nu_{0}; s \rangle \psi (0, \cdot) } \\
&= \int_0^\tau \intO{ \Big[ \langle \nu_{t,x}; s \rangle \partial_t \psi + \langle \nu_{t,x}; s \vc{v} \rangle \cdot \Grad \psi \Big] } \ \dt
+ \int_0^\tau \langle r^C; \Grad \psi \rangle \ \dt.
\end{split}
\end{equation}

\color{black}

\item {\bf Momentum equation.}
\[
\vu = \left< \nu_{t,x}; \vc{v} \right> \in L^2(0,T; W^{1,2}_0 (\Omega;R^N)),
\]
and there exists a measure $r^M\in L^1([0,T];\mathcal{M}(\Ov{\Omega}))$ and $\xi\in L^1(0,T)$ such that for a.a.\ $\tau\in(0,T)$ and every $\varphi \in C^1([0,T] \times \Ov{\Omega}; R^N)$,  $\varphi|_{\partial \Omega} = 0$,
\begin{equation}
\left| \langle r^M (\tau) ; \Grad\varphi \rangle \right| \leq \xi(\tau)  \mathcal{D} (\tau) \| \varphi \|_{C^1(\Ov{\Omega})}
\end{equation}
and
\begin{equation} \label{dmvB2}
\begin{split}
&\intO{ \langle \nu_{\tau,x}; s \vc{v} \rangle \cdot \varphi (\tau, \cdot) }  -  \intO{ \langle \nu_{0}; s \vc{v} \rangle \cdot \varphi (0, \cdot) }\\
&= \int_0^\tau \intO{ \Big[ \langle \nu_{t,x} ; s \vc{v} \rangle \cdot \partial_t \varphi + \langle \nu_{t,x}; s (\vc{v} \otimes \vc{v} ) \rangle : \Grad \varphi +
\langle \nu_{t,x} ; p(s) \rangle \Div \varphi \Big] } \ \dt\\
& - \int_0^\tau \intO{  \mathbb{S}({\Grad \vu })  : \Grad \varphi } \ \dt + \int_0^\tau \left< {r}^M ; \Grad \varphi \right> \ \dt.
\end{split}
\end{equation}

\color{black}

\item{\bf Energy inequality.}
\begin{equation} \label{dmvEI}
\begin{split}
\intO{ \left< \nu_{\tau,x};  \left( \frac{1}{2} s |\vc{v}|^2 + P(s) \right) \right> }
&+ \int_0^\tau \intO{ \mathbb{S}(\Grad \vu) : \Grad \vu  } \ \dt + \mathcal{D}(\tau) \\
&\leq \intO{ \left< \nu_0; \left( \frac{1}{2} s |\vc{v}|^2 + P(s) \right) \right>  }
\end{split}
\end{equation}
for a.a. $\tau \in (0,T)$.
In addition, the following version
of ``Poincar\' e's inequality" holds for a.a. $\tau \in (0,T)$:
\begin{equation} \label{KoPo}
\int_0^\tau \intO{ \left< \nu_{t,x} ;  |\vc{v} - \vu|^2 \right> } \ \dt \leq c_P  \mathcal{D}(\tau).
\end{equation}

\end{itemize}
{
\begin{Remark}
Hypothesis \eqref{KoPo} is motivated by the following observation: Suppose that
$$
\vu^\varepsilon\to \vu\ \mbox{weakly in } \ L^2(0,T;W^{1,2}_0(\Omega;R^N)),
$$
then
\begin{equation*}
\begin{split}
\int_0^\tau \intO{ \left< \nu_{t,x} ;  |\vc{v} - \vu|^2 \right> } \ \dt&=\lim\limits_{\varepsilon\to0}
\int_0^\tau\intO{|\vu^\varepsilon-\vu|^2}\ \dt\le
c_P\lim\limits_{\varepsilon\to0}\int_0^\tau\intO{|\nabla \vu^\varepsilon-\nabla\vu|^2}\ \dt\\
&=c_P\lim\limits_{\varepsilon\to0}\int_0^\tau\intO{|\nabla \vu^\varepsilon|^2-|\nabla\vu|^2}\ \dt
\leq c_P  \mathcal{D}(\tau),
\end{split}
\end{equation*}
provided the dissipation defect $\mathcal{D}$ ``contains" the oscillations and concentrations in the velocity gradient.
\end{Remark}
}
\end{Definition}

\color{black}

We tacitly assume that all integrals in (\ref{dmvB1}--\ref{KoPo}) are well defined, meaning, all integrands are measurable and at least integrable.


Notice that $\mathbb{S}(\Grad \vu): \Grad \vu \geq 0$ so that the dissipative term in the energy inequality is nonnegative.

\color{black}

The function $\mathcal{D}$ represents a dissipation defect usually attributed to (hypothetical) singularities that may appear in the course of the fluid evolution. The measure-valued formulation contains a minimal piece of information encoded in system (\ref{I1}--\ref{I3}). In contrast with the definition introduced by
Neustupa \cite{Neustup}, the oscillatory and concentration components are clearly separated and, more importantly, the energy balance is included as an integral part of the present approach.

Although one often uses the framework of Alibert and Bouchitt\'e~\cite{AlBo} in order to handle oscillations and concentrations (for instance in \cite{Gwi, BrDLSz, GwSwWi}), we choose here to give a somewhat simpler representation of the concentration effects, thereby avoiding usage of the concentration-angle measure. Indeed, the generalized Young measures of Alibert-Bouchitt\'e capture information on \emph{all} nonlinear functions of the generating sequence with suitable growth, whereas for our purposes this information is not fully needed, as we deal only with specific nonlinearities (such as $\rho\vc{u}\otimes\vc{u}$, $|\Grad\vc{u}|^2$, etc.), which are all encoded in the dissipation defect $\mathcal{D}$. This approach is inspired by~\cite{DeStTz}. We feel that the present formulation improves readability, and, more importantly,
extends considerably the class of possible applications of the weak-strong uniqueness principle stated below. Indeed, it is possible to define dissipative measure-valued solutions in the framework of Alibert-Bouchitt\'e and show that they give rise to a dissipative measure-valued solutions as defined above, but presumably not vice versa. Let us also point out that an analogue of our dissipative measure-valued solutions could be considered also for the incompressible and compressible Euler system and might thus lead to a slight simplification and generalization of the results in~\cite{BrDLSz} and \cite{GwSwWi}.

\color{black}

The considerations of Section~\ref{BM} immediately yield the following existence result: \color{black}
\begin{Theorem}
Suppose $\Omega$ is a regular bounded domain in $R^2$ or $R^3$, and suppose the pressure satisfies~\eqref{BB6}. If $(\rho_0,\vc{u}_0)$ is initial data with finite energy, then there exists a dissipative measure-valued solution with initial data
\begin{equation}
\nu_0=\delta_{(\rho_0,\vc{u}_0)}.
\end{equation}
\end{Theorem}
\begin{proof}
For every $K>0$, we find a weak solution to Brenner's model with initial data $\vc{u}_0^K\in C_c^{\infty}(\Omega)$ and $\rho^K_0\in C^\infty(\Ov{\Omega})$ such that $\Grad\rho^K_0\cdot\vc{n}|_{\partial \Omega} = 0$ and such that
$\rho_0^K\to\rho_0$, $\rho_0^K\vc{u}^K_0\to\rho_0\vc{u}_0$, and
\[
\frac{1}{2}\rho_0^K|\vc{u}^K_0|^2+P(\rho_0^K)\to\frac{1}{2}\rho_0|\vc{u}_0|^2+P(\rho_0)
\]
in $L^1(\Omega)$, respectively. Indeed, it is easy to see that such an approximation of the initial density exists (use a simple truncation and smoothing argument). Then, the arguments of Section~\ref{BM} yield a dissipative measure-valued solution with
\[
\mathcal{D}(\tau)=E_\infty(\tau)[\Ov{\Omega}]+\sigma_\infty[[0,\tau]\times\Ov{\Omega}]
\]
for a.a.\ $\tau\in(0,T)$. Moreover, we have $r^C=0$ and $\chi\equiv0$, $\xi\equiv c$. The Poincar\'e-Korn inequality~\eqref{KoPo} is an easy consequence of the respective inequality for each $\vc{u}^K$.
\end{proof}
Note that our definition of dissipative measure-valued solutions is arguably broader than necessary: For instance, any approximation sequence with a uniform bound on the energy will not concentrate in the momentum, whence $r^C=0$. We choose to include such an effect in our definition anyway since even in this potentially larger class of measure-valued solutions we can still show weak-strong uniqueness: a measure-valued and a smooth solution starting from the same initial data coincide as long as the latter exists. In other words, the set of classical (smooth) solutions is stable in the class of dissipative measure-valued solutions. Showing this property is the main goal of the present paper.

\section{Relative energy}

The commonly used form of the relative energy (entropy) functional in the context of weak solutions to the barotropic Navier-Stokes system reads
\[
\mathcal{E} \left( \vr, \vu \ \Big| r, \vc{U} \right)  = \intO{ \left[ \frac{1}{2} \vr |\vu - \vc{U}|^2 + P(\vr) - P'(r) (\vr - r) - P(r) \right] },
\]
where $\vr$, $\vu$ is a weak solution and $r$ and $\vc{U}$ are arbitrary ``test'' functions mimicking the basic properties of $\vr$, $\vu$, specifically,
$r$ is positive and $\vc{U}$ satisfies the relevant boundary conditions, see Feireisl et al.\ \cite{FeJiNo}, Germain \cite{Ger}, Mellet and Vasseur \cite{MeVa1}, among others. Here, the
crucial observation is that
\[
\mathcal{E} \left( \vr, \vu \ \Big| r, \vc{U} \right) = \intO{ \left[ \frac{1}{2} \vr |\vu|^2 + P(\vr) \right] } - \intO{ \vr \vu \cdot \vc{U} } + \intO{ \frac{1}{2} \vr |\vc{U}|^2 }
- \intO{ P'(r) \vr } + \intO{ p(r) },
\]
where all integrals on the right-hand side may be explicitly expressed by means of either the energy inequality or the field equations. Accordingly, a relevant
candidate in the framework of (dissipative) measure valued solutions is
\[
\begin{split}
&\mathcal{E}_{mv} \left( \vr, \vu \ \Big| r, \vc{U} \right)(\tau)  = \intO{ \left< \nu_{\tau,x};  \frac{1}{2} s |\vc{v} - \vc{U}|^2 + P(s) - P'(r) (s - r) - P(r) \right> } \\
&= \intO{ \left< \nu_{\tau,x}; \frac{1}{2} s |\vc{v}|^2 + P(s) \right> }  - \intO{ \left< \nu_{\tau,x}; s \vc{v} \right> \cdot \vc{U} } + \intO{ \frac{1}{2}
\left< \nu_{\tau,x} ; s \right> |\vc{U}|^2 }\\
&- \intO{ \left< \nu_{\tau,x} ; s \right> P'(r) } + \intO{ p(r) }.
\end{split}
\]
Our goal in the remaining part of this section is to express all integrals on the right hand side in terms of the energy balance (\ref{dmvEI})  and the field equations
(\ref{dmvB1}), (\ref{dmvB2}).

\subsection{Density dependent terms}

Using the equation of continuity (\ref{dmvB1}) with test function $\frac{1}{2}|\vc{U}|^2$, we get
\begin{equation} \label{RE1}
\begin{split}
&\intO{ \frac{1}{2} \left< \nu_{\tau,x}; s \right> |\vc{U}|^2(\tau, \cdot) } - \intO{ \frac{1}{2} \left< \nu_{0,x}; s \right> |\vc{U}|^2(0, \cdot) } \\
& = \int_0^\tau \intO{ \left[ \left< \nu_{t,x}; s \right> \vc{U} \cdot \partial_t \vc{U} + \left< \nu_{t,x}; s \vc{v} \right> \cdot \vc{U}
\cdot \Grad \vc{U} \right] } \ \dt + \int_0^\tau \intO{ \left< r^C; \frac{1}{2} \Grad |\vc{U}|^2 \right> } \ \dt
\end{split}
\end{equation}
provided $\vc{U} \in C^1([0,T] \times \Ov{\Omega}; R^N)$. \color{black}

Similarly, testing with $P'(r)$ we can write
\begin{equation} \label{RE2}
\begin{split}
& \intO{ \left< \nu_{\tau,x} ;s \right>  P'(r) (\tau, \cdot) }  -  \intO{ \left< \nu_{0,x} ;s \right>  P'(r) (0, \cdot) }\\
&= \int_0^\tau \intO{ \left[ \left< \nu_{t,x}; s \right> P''(r)  \partial_t r + \left< \nu_{t,x}; s \vc{v} \right> \cdot P''(r)
\cdot \Grad r \right] } \ \dt + \int_0^\tau \intO{ \left< r^C ; P'(r)  \right> } \ \dt\\
&= \int_0^\tau \intO{ \left[ \left< \nu_{t,x}; s \right> \frac{p'(r)}{r}  \cdot \partial_t r + \left< \nu_{t,x}; s \vc{v} \right> \frac{p'(r)}{r} \cdot
\Grad r \right] } \ \dt   + \int_0^\tau \intO{ \left< r^C ; \Grad P'(r)  \right> } \ \dt
\end{split}
\end{equation} \color{black}
provided $r > 0$ and $r \in C^1([0,T] \times \Ov{\Omega})$, and $P$ is twice continuously differentiable in $(0, \infty)$.

\subsection{Momentum dependent terms}


Analogously to the preceding part, we use (\ref{dmvB2}) to compute 
\begin{equation} \label{RE3}
\begin{split}
&\intO{ \left< \nu_{\tau,x} ; s \vc{v} \right> \cdot \vc{U} (\tau, \cdot) } - \intO{ \left< \nu_{0,x} ; s \vc{v} \right> \cdot \vc{U} (0, \cdot) } \\
&= \int_0^\tau \intO{ \left< \nu_{t,x}; s \vc{v} \right> \cdot \partial_t \vc{U} } \ \dt + \int_0^\tau \int_{\Ov{\Omega}} \left[ \left< \nu_{t,x}; s \vc{v} \otimes \vc{v} \right> : \Grad \vc{U} + \left< \nu_{t,x}; p(s) \right> \Div \vc{U} \right] \dx \ \dt\\
& - \int_0^\tau \intO{ \left< \nu_{t,x} ; \mathbb{S} (\mathbb{D} ) \right> : \Grad \vc{U} } \ \dt + \int_0^\tau \left< r^M; \Grad \vc{U} \right>  \ \dt\\
&= \int_0^\tau \intO{ \left< \nu_{t,x}; s \vc{v} \right> \cdot \partial_t \vc{U} } \ \dt + \int_0^\tau \int_{\Ov{\Omega}} \left[ \left< \nu_{t,x}; s \vc{v} \otimes \vc{v} \right> : \Grad \vc{U} + \left< \nu_{t,x}; p(s) \right> \Div \vc{U} \right] \dx \ \dt\\
& - \int_0^\tau \intO{ \mathbb{S}(\Grad \vu)  : \Grad \vc{U} } \ \dt + \int_0^\tau \left< r^M; \Grad \vc{U} \right>  \ \dt
\end{split}
\end{equation}
for any $\vc{U} \in C^1([0,T] \times \Ov{\Omega}; R^N)$, $\vc{U}|_{\partial \Omega} = 0$.

\color{black}

\subsection{Relative energy inequality}

Summing up the previous discussion we may deduce a measure-valued analogue of the relative energy inequality:
\begin{equation} \label{RE5}
\begin{split}
\mathcal{E}_{mv} \left( \vr, \vu \ \Big| r , \vc{U} \right)    &+
\int_0^\tau \mathbb{S}(\Grad \vu) : \left( \Grad \vu
- \Grad \vc{U} \right)   \ \dxdt  + \mathcal{D}(\tau) \\
&\leq \intO{ \left< \nu_{0,x};  \frac{1}{2} s |\vc{v} - \vc{U}_0|^2 + P(s) - P'(r_0) (s - r_0) - P(r_0) \right> } \\
& - \int_0^\tau \intO{ \left< \nu_{t,x}, s \vc{v} \right> \cdot \partial_t \vc{U} } \ \dt\\
& - \int_0^\tau \int_{\Ov{\Omega}} \left[ \left< \nu_{t,x};  s \vc{v} \otimes \vc{v} \right> : \Grad \vc{U} + \left< \nu_{t,x}; p(s) \right> \Div \vc{U} \right] \dx \ \dt\\
& + \int_0^\tau \intO{ \left[ \left< \nu_{t,x}; s \right> \vc{U}  \cdot \partial_t \vc{U} + \left< \nu_{t,x}; s \vc{v} \right> \cdot \vc{U}
\cdot \Grad \vc{U} \right] } \ \dt\\
&+ \int_0^\tau \intO{ \left[ \left< \nu_{t,x} ; \left(1 - \frac{s}{r}  \right) \right> p'(r) \partial_t r - \left< \nu_{t,x}; s \vc{v} \right> \cdot \frac{p'(r)}{r}
\Grad r \right] } \ \dt\\
&+ \int_0^\tau  \left< r^C; \frac{1}{2} \Grad |\vc{U}|^2  - \Grad P'(r) \right>  \  \dt - \int_0^\tau  \left< r^M ; \Grad \vc{U} \right>  \dt.
\end{split}
\end{equation}
\color{black}
As already pointed out,
the relative entropy inequality (\ref{RE5}) holds for any $r \in C^1([0,T] \times \Ov{\Omega})$, $r > 0$, and any
$\vc{U} \in C^1([0,T] \times \Ov{\Omega}; R^N)$,
$\vc{U}|_{\partial \Omega} = 0$.

Moreover,
in accordance with Definition \ref{DD1}, we have
\[
\begin{split}
&\left| \int_0^\tau  \left< r^C; \frac{1}{2} \Grad |\vc{U}|^2  - \Grad P'(r) \right>  \  \dt - \int_0^\tau  \left< r^M ; \Grad \vc{U} \right>  \dt \right| \\
&\leq c \left( \left\| \Grad \vc{U} \right\|_{C([0, T] \times \Ov{\Omega}; R^{N \times N})} +
\left\| r \right\|_{C([0,T] \times \Ov{\Omega}) } + \left\| \Grad r \right\|_{C([0, T] \times \Ov{\Omega}; R^{N})}    \right) \int_0^\tau (\chi(t)+\xi(t)) \mathcal{D}(t) \ \dt.
\end{split}
\]
Thus the validity of (\ref{RE5}) can be extended to the following class of test functions by a simple argument: \color{black}
\begin{equation} \label{TEST}
\vc{U} , \Grad \vc{U}, r, \Grad r \in C([0,T] \times \Ov{\Omega}),\ \partial_t r, \ \partial_t \vc{U} \in L^1(0,T; C(\Ov{\Omega})),\
r > 0,\ \vc{U}|_{\partial \Omega} = 0.
\end{equation}

\section{Weak-strong uniqueness}

Now, we suppose that the test functions $r$, $\vc{U}$ belong to the class (\ref{TEST}), and, in addition, solve the Navier-Stokes system (\ref{I1}--\ref{I3}).
Our goal is to show that the measure valued solution and the strong one are close in terms of the ``distance'' of the initial data. We proceed in several steps.

\subsection{Continuity equation}

In addition to the general hypotheses that guarantee (\ref{RE5}) suppose that $r, \vc{U}$ satisfy the equation of continuity
\bFormula{RE6}
\partial_t r + \Div (r \vc{U} ) = 0.
\eF

Accordingly, we may rewrite (\ref{RE5}) as
\begin{equation} \label{RE7}
\begin{split}
\mathcal{E}_{mv} \left( \vr, \vu \ \Big| r , \vc{U} \right)    &+
\int_0^\tau \intO{ \mathbb{S}(\Grad \vu) : \left( \Grad \vu
- \Grad \vc{U} \right)  } \ \dt  + \mathcal{D}(\tau) \\
&\leq \intO{ \left< \nu_{0,x};  \frac{1}{2} s |\vc{v} - \vc{U}_0|^2 + P(s) - P'(r_0) (s - r_0) - P(r_0) \right> } \\
& - \int_0^\tau \intO{ \left< \nu_{t,x}, s \vc{v} \right> \cdot \partial_t \vc{U} } \ \dt - \int_0^\tau \int_{\Ov{\Omega}}  \left< \nu_{t,x};  s \vc{v} \otimes \vc{v} \right> : \Grad \vc{U}  \dx \ \dt\\
& + \int_0^\tau \intO{ \left[ \left< \nu_{t,x}; s \right> \vc{U}  \cdot \partial_t \vc{U} + \left< \nu_{t,x}; s \vc{v} \right> \cdot \vc{U}
\cdot \Grad \vc{U} \right] } \ \dt\\
& + \int_0^\tau \intO{  \left< \nu_{t,x}; s \vc{U} - s \vc{v} \right> \cdot \frac{p'(r)}{r} \Grad r  } \ \dt \\
&- \int_0^\tau \intO{  \left< \nu_{t,x} ; p(s) - p'(r)(s -r) - p(r)  \right> \Div \vc{U} } \ \dt\\
&+ c \int_0^\tau  (\chi(t)+\xi(t)) \mathcal{D}(t)  \dt
\end{split}
\end{equation}
where the constant $c$ depends only on the norms of the test functions specified in (\ref{TEST}).

\subsection{Momentum equation}

In addition to (\ref{RE6}) suppose that $r, \vc{U}$ satisfy also the momentum balance
\[
\partial_t \vc{U} + \vc{U} \cdot \Grad \vc{U} + \frac{1}{r} \Grad p(r) = \frac{1}{r} \Div \mathbb{S} (\Grad \vc{U}).
\]
Indeed, it is easily seen that this follows from the momentum equation in conjunction with~\eqref{RE6}. Accordingly, relation (\ref{RE7}) reduces to
\begin{equation} \label{RE8}
\begin{split}
\mathcal{E}_{mv} \left( \vr, \vu \ \Big| r , \vc{U} \right)    &+
\int_0^\tau \intO{ \mathbb{S}(\Grad \vu) : \left( \Grad \vu
- \Grad \vc{U} \right)  } \ \dt   + \mathcal{D}(\tau) \\
&\leq \intO{ \left< \nu_{0,x};  \frac{1}{2} s |\vc{v} - \vc{U}_0|^2 + P(s) - P'(r_0) (s - r_0) - P(r_0) \right> } \\
& + \int_0^\tau \intO{ \left< \nu_{t,x}, s \vc{v}  - s \vc{U} \right> \cdot \Grad \vc{U} \cdot \vc{U} } \ \dt - \int_0^\tau \int_{\Ov{\Omega}}  \left< \nu_{t,x};  s \vc{v} \otimes \vc{v} \right> : \Grad \vc{U}  \dx \ \dt\\
& + \int_0^\tau \intO{ \left< \nu_{t,x}; s \vc{v} \right> \cdot
 \Grad \vc{U}\cdot  \vc{U}} \ \dt\\
& + \int_0^\tau \intO{  \left< \nu_{t,x}; s \vc{U} - s \vc{v} \right> \cdot \frac{1}{r} \Div \mathbb{S} (\Grad \vc{U})  } \ \dt \\
&+ \int_0^\tau \intO{  \left< \nu_{t,x} ; p(s) - p'(r)(s -r) - p(r)  \right> \Div \vc{U} } \ \dt\\
&+ c \int_0^\tau  (\chi(t)+\xi(t)) \mathcal{D}(t)  \dt
\end{split}
\end{equation}
where, furthermore, \color{black}
\[
\begin{split}
&\int_0^\tau \intO{ \left< \nu_{t,x}; \left( s \vc{v} - s \vc{U} \right) \right> \cdot \Grad \vc{U} \cdot \vc{U} } \ \dt
- \int_0^\tau \int_{{\Omega}} \left< \nu_{t,x} ; s \vc{v} \otimes \vc{v} \right> : \Grad \vc{U} \dx\ \dt\\
&+ \int_0^\tau \intO{ \left< \nu_{t,x}; s \vc{v} \right> \cdot \Grad \vc{U} \cdot \vc{U}   } \ \dt\\
&=  \int_0^\tau \intO{ \left< \nu_{t,x}; \left( s \vc{v} - s \vc{U} \right) \right> \cdot \Grad \vc{U} \cdot \vc{U} } \ \dt +
\int_0^\tau \int_{{\Omega}} \left< \nu_{t,x}; s \vc{v} \cdot   \Grad \vc{U}\cdot(\vc{U} - \vc{v} ) \right> \dx\ \dt\\
&= \int_0^\tau \int_{{\Omega}} \left< \nu_{t,x}; s (\vc{v} - \vc{U}) \cdot \Grad \vc{U}\cdot (\vc{U} - \vc{v} )  \right>  \dx\ \dt.
\end{split}
\]
Thus, finally, (\ref{RE8}) can be written as
\begin{equation} \label{RE9}
\begin{split}
\mathcal{E}_{mv} \left( \vr, \vu \ \Big| r , \vc{U} \right)    &+
\int_0^\tau \intO{  \mathbb{S}(\Grad \vu -\Grad\vc{U}):(\Grad \vu -\Grad\vc{U})  } \ \dt+ \mathcal{D}(\tau) \\
&\leq \intO{ \left< \nu_{0,x};  \frac{1}{2} s |\vc{v} - \vc{U}_0|^2 + P(s) - P'(r_0) (s - r_0) - P(r_0) \right> } \\
& + \int_0^\tau \int_{{\Omega}} \left< \nu_{t,x}; s (\vc{v} - \vc{U})\cdot \Grad \vc{U} \cdot (\vc{U} - \vc{v} )  \right>  \dx\ \dt\\
& - \int_0^\tau \intO{ (\Grad \vu  - \Grad \vc{U} )  : \mathbb{S} (\Grad \vc{U})  } \ \dt\\
& + \int_0^\tau \intO{  \left< \nu_{t,x}; s \vc{U} - s \vc{v} \right> \cdot \frac{1}{r} \Div \mathbb{S} (\Grad \vc{U})  } \ \dt \\
&+ \int_0^\tau \intO{  \left< \nu_{t,x} ; p(s) - p'(r)(s -r) - p(r)  \right> \Div \vc{U} } \ \dt\\
&+ c \int_0^\tau  (\chi(t)+\xi(t)) \mathcal{D}(t)  \dt.
\end{split}
\end{equation}

\color{black}

\subsection{Compatibility}

Our last goal is to handle the difference
\[
\int_0^\tau \intO{  \left< \nu_{t,x}; s \vc{U} - s \vc{v} \right> \cdot \frac{1}{r} \Div \mathbb{S} (\Grad \vc{U})  } \ \dt
- \int_0^\tau \intO{ (\Grad \vu  - \Grad \vc{U} )  : \mathbb{S} (\Grad \vc{U})  } \ \dt.
\]
To this end, we need slightly more regularity than required in (\ref{TEST}), namely
\[
\Div \tn{S}(\Grad \vc{U}) \in L^2(0,T; L^\infty({\Omega}; R^{N\times N})) \ \mbox{or,
equivalently}\ \partial_t \vc{U} \in L^2(0,T; L^\infty ({\Omega}; R^N)).
\]

\color{black}

Now,
since
\[
{\vu}\in L^2((0,T); W^{1,2}_0(\Omega;R^N)),
\]
we get
\[
\begin{split}
&\int_0^\tau \intO{  \left< \nu_{t,x}; s \vc{U} - s \vc{v} \right> \cdot \frac{1}{r} \Div \mathbb{S} (\Grad \vc{U})  } \ \dt
- \int_0^\tau \intO{  (\Grad \vu - \Grad \vc{U} )  : \mathbb{S} (\Grad \vc{U})  } \ \dt\\
& = \int_0^\tau \intO{ \left< \nu_{t,x};  \left( s \vc{U} - s \vc{v} + r \vc{v} - r \vc{U} \right) \right> \cdot \frac{1}{r} \Div \mathbb{S} (\Grad \vc{U}) } \ \dt\\
& = \int_0^\tau \intO{ \left< \nu_{t,x};  (s - r) (\vc{U} - \vc{v} ) \right> \cdot \frac{\Div \mathbb{S} (\Grad \vc{U})}{r} } \ \dt.
\end{split}
\]

\color{black}

Now, we write
\[
\begin{split}
&\left< \nu_{t,x};  (s - r) (\vc{U} - \vc{v} ) \right> \\
&= \left< \nu_{t,x} ; \psi (s)  (s - r)  (\vc{U} - \vc{v} ) \right>  + \left< \nu_{t,x}; (1 - \psi (s)) (s - r) (\vc{U} - \vc{v} ) \right>,
\end{split}
\]
where
\[
\psi \in C^\infty_c(0, \infty),\ 0 \leq \psi \leq 1, \ \psi(s) = 1 \ \mbox{for} \ s \in ( \inf{r}, \sup{r} ).
\]
Consequently, we get
\[
\left< \nu_{t,x};  \psi (s) (s - r) (\vc{U} - \vc{v} ) \right>
\leq \frac{1}{2}\left< \nu_{t,x};  \frac{\psi (s)^2}{\sqrt{s}} (s - r)^2 \right>  + \frac{1}{2}\left< \nu_{t,x};  \frac{\psi (s)^2}{\sqrt{s}} {s} |\vc{U} - \vc{v} |^2 \right>,
\]
where, as $\psi$ is compactly supported in $(0, \infty)$, both terms can be controlled in (\ref{RE9}) by $\mathcal{E}_{mv}$, as is easily verified.

Next, we write
\[
\begin{split}
&\left< \nu_{t,x}; (1 - \psi (s)) (s - r) (\vc{U} - \vc{v} ) \right> \\
& = \left< \nu_{t,x}; \omega_1 (s) (s - r) (\vc{U} - \vc{v} ) \right> + \left< \nu_{t,x};  \omega_2 (s) (s - r) (\vc{U} - \vc{v} ) \right>,
\end{split}
\]
where
\[
{\rm supp}[\omega_1] \subset [0, \inf r), \ {\rm supp}[\omega_2] \subset (\sup r, \infty], \ \omega_1+\omega_2=1-\psi.
\]
Accordingly,
\[
\left< \nu_{t,x};  \omega_1 (s) (s - r) (\vc{U} - \vc{v} )  \right> \leq c(\delta) \left< \nu_{t,x} ; \omega_1^2(s) (s -r)^2 \right>  + \delta \left< \nu_{t,x}; |\vc{U} - \vc{v}|^2 \right>
\]
where the former term on the right-hand side is controlled by $\mathcal{E}_{mv}$ while the latter can be absorbed by the left hand side of (\ref{RE9}) by virtue of
Poincar\' e inequality stipulated in (\ref{KoPo}) provided $\delta > 0$ has been chosen small enough. Indeed, on one hand,
\[
\begin{split}
\left< \nu_{t,x}; |\vc{U} - \vc{v} |^2 \right> &= |\vc{U}|^2 - 2 \vu \cdot \vc{U} + |\vu|^2 + \left< \nu_{t,x} ; |\vc{v}|^2 - |\vu|^2 \right>\\
& = |\vc{u} - \vc{U} |^2  + \left< \nu_{t,x} ; |\vc{v} - \vu|^2 \right>;
\end{split}
\]
whence, by virtue of (\ref{KoPo}) and the standard Poincar\' e-Korn inequality,
\[
\int_0^\tau \intO{ \left< \nu_{t,x}; |\vc{U} - \vc{v} |^2 \right> } \ \dt \leq c_P \left( \int_0^\tau \intO{ \left( \mathbb{S} (\Grad \vu -
\Grad \vc{U}) \right) : \left(\Grad \vu -
\Grad \vc{U} \right)  } \ \dt + \mathcal{D}(\tau)
\right).
\]


\color{black}

Finally,
\[
\left< \omega_2 (s) (s - r) (\vc{U} - \vc{v} ) \right> \leq c \ \left< \nu_{t,x}; \omega_2 (s) \left( s +  s (\vc{U} -  \vc{v} )^2 \right) \right>,
\]
where both integrals are controlled by $\mathcal{E}_{mv}$.


Summing up the previous discussion, we deduce from (\ref{RE9}) that
\[
\begin{split}
&\mathcal{E}_{mv} \left( \vr, \vu \ \Big| r , \vc{U} \right)
  + \frac{1}{2} \mathcal{D}(\tau)   \\
&\leq \intO{ \left< \nu_{0,x};  \frac{1}{2} s |\vc{v} - \vc{U}(0, \cdot) |^2 + P(s) - P'(r(0,\cdot)) (s - r(0,\cdot)) - P(r(0,\cdot)) \right> }\\
&+ c \left( \int_0^\tau   \mathcal{E}_{mv} \left( \vr, \vu \ \Big| r , \vc{U} \right)\ \dt + \int_0^\tau (\chi(t)+\xi(t))\mathcal{D}(t) \ \dt \right).
\end{split}
\]
Thus applying Gronwall's lemma, we conclude that
\begin{equation} \label{RE10}
\begin{split}
&\mathcal{E}_{mv}  \left( \vr, \vu \ \Big| r , \vc{U} \right) (\tau)  +\mathcal{D}(\tau) \\
&\leq c(T) \intO{ \left< \nu_{0,x};  \frac{1}{2} s |\vc{v} - \vc{U}(0, \cdot) |^2 + P(s) - P'(r(0,\cdot)) (s - r(0,\cdot)) - P(r(0,\cdot)) \right> }
\end{split}
\end{equation}
for a.a. $\tau \in [0,T]$.

\color{black}

We have shown the main result of the present paper.


\begin{Theorem} \label{TT1}
Let $\Omega \subset R^N$, $N=2,3$ be a bounded smooth domain. Suppose the pressure $p$ satisfies (\ref{BB6}). Let $\{ \nu_{t,x}, \mathcal{D} \}$ be a dissipative measure-valued
solution to the barotropic Navier-Stokes system (\ref{I1}--\ref{I3}) in $(0,T) \times \Omega$, with the initial state represented by $\nu_0$, in the sense specified in Definition \ref{DD1}.
Let $[r, \vc{U}]$ be a strong solution of (\ref{I1}--\ref{I3}) in $(0,T) \times \Omega$ belonging to the class
\[
r, \ \Grad r, \ \vc{U},\ \Grad \vc{U} \in C([0,T] \times \Ov{\Omega}),\
\partial_t \vc{U} \in L^2(0,T; C(\overline{\Omega};R^N)),\ r > 0,\ \vc{U}|_{\partial \Omega} = 0.
\]

Then there is a constant $\Lambda = \Lambda(T)$, depending only on the norms of $r$, $r^{-1}$, $\vc{U}$, $\chi$, and $\xi$ in the aforementioned spaces, such that
\[
\begin{split}
& \intO{ \left< \nu_{\tau,x};  \frac{1}{2} s |\vc{v} - \vc{U}|^2 + P(s) - P'(r) (s - r) - P(r) \right> }
 \\ &+  \int_0^\tau \intO{ | \Grad \vu - \Grad \vc{U} |^2   } \ \dt + \mathcal{D}(\tau) \\
&\leq \Lambda(T) \intO{ \left< \nu_{0,x};  \frac{1}{2} s |\vc{v} - \vc{U}(0, \cdot) |^2 + P(s) - P'(r(0,\cdot)) (s - r(0,\cdot)) - P(r(0,\cdot)) \right> }
\end{split}
\]
for a.a. $\tau \in (0,T)$. In particular, if the initial states coincide, meaning
\[
\nu_{0,x} = \delta_{[ r(0,x), \vc{U}(0,x) ]} \ \mbox{for a.a.} \ x \in \Omega
\]
then $\mathcal{D} = 0$, and
\[
\nu_{\tau,x} = \delta_{[ r(\tau,x), \vc{U}(\tau,x) ]} \ \mbox{for a.a.}\ \tau \in (0,T),\ x \in \Omega.
\]

\end{Theorem}

\color{black}

\section{Examples of problems generating measure-valued solutions}
\label{E}

Besides the model of Brenner discussed in Section \ref{I},
there is a vast class of problems -- various approximations of the barotropic Navier-Stokes system (\ref{I1}--\ref{I3}) -- generating (dissipative) measure-valued solutions.
Below, we mention three examples among many others.

\subsection{Artificial pressure approximation}

The theory of weak solutions proposed by Lions \cite{LI4} and later developed in \cite{FNP} does not cover certain physically interesting cases. For the sake of simplicity,
consider the pressure $p$ in its iconic form
\[
p(\vr) = a \vr^\gamma,\ a > 0,\ \gamma \geq 1
\]
obviously satisfying (\ref{BB6}). The existence of weak solutions is known in the following cases:
\[
N = 2,\ \gamma \geq 1\  \mbox{and }\ N = 3,\ \gamma > \frac{3}{2}.
\]
Note that the critical case $\gamma = 1$ for $N = 2$ was solved only recently by Plotnikov and Weigant \cite{PloWei}.

This motivates the following approximate problem
\begin{eqnarray}
\label{AI1}
\partial_t \vr + \Div (\vr \vu) &=& 0,  \\
\label{AI2}  \partial_t (\vr \vu) + \Div (\vr \vu \otimes \vu) + \Grad p(\vr) + \delta \Grad \vr^\Gamma &=& \Div \mathbb{S} (\Grad \vu), \\
\label{AI3} \vu|_{\partial \Omega} &=& 0,
\end{eqnarray}
where $\delta > 0$ is a small parameter and $\Gamma > 1$ is large enough to ensure the existence of weak solutions.

Repeating the arguments applied in Section \ref{BM} to Brenner's model, it is straightforward to check that a family of weak solutions $\{ \vr_\delta, \vu_{\delta} \}_{\delta > 0}$,
satisfying the energy inequality,
generates a dissipative measure-valued solution in the sense of Definition \ref{DD1}. Indeed it is enough to observe that
\[
p(\vr) + \delta \vr^\Gamma \leq c \left( P(\vr) + \frac{\delta}{\Gamma - 1} \vr^{\Gamma} \right) \ \mbox{for all}\ \vr \geq 1,
\]
where the constant is uniform with respect to $\delta \to 0$.

We conclude that the weak solutions of the problem with vanishing artificial pressure generate a dissipative measure-valued solution. In particular, as a consequence of
Theorem \ref{TT1}, they converge to the (unique) strong solution provided it exists and the initial data are conveniently adjusted. We remark that strong
solutions to the barotropic Navier-Stokes system exist at least locally in time provided
\begin{itemize}
\item
the pressure is a sufficiently smooth function of $\vr$,
\item
the domain $\Omega$ has a regular boundary, and
\item
the initial data are smooth enough and satisfy the necessary compatibility conditions as the case may be,
\end{itemize}
see e.g.\ Cho, Choe and Kim \cite{ChoChoeKim}, Valli and Zaj{\c{a}}czkowski \cite{VAZA}.

\subsection{Multipolar fluids}

The theory of multipolar fluid was developed by Ne\v cas and \v Silhav\' y \cite{NESI} in the hope to develop an alternative approach to regularity for
compressible fluids. The problems may take various forms depending on the shape of the viscous stress
\[
\mathbb{T}(\vu, \Grad \vu, \ \Grad^2 \vu, \dots) = \mathbb{S}( \Grad \vu) + \delta \sum_{j = 1}^{k-1} \left( (-1)^j \mu_j \Delta^j
(\Grad \vu + \Grad^t \vu) + \lambda_j \Delta^j \Div \vu \ \mathbb{I} \right) + \ \mbox{non-linear terms.}
\]
The resulting system has a nice variational structure and, for $k$ large enough, admits global in time smooth solutions, see Ne\v cas, Novotn\' y and \v Silhav\' y \cite{NeNoSil}.

It is natural to conjecture that the (smooth) solutions of the multipolar system will converge to their weak counterparts as $\delta \to 0$ at least in the cases where the
pressure complies with the requirements of Lions' theory. However, this is to the best of our knowledge an open problem. Instead, such a process may and does generate a
(dissipative) measure valued solutions at least for a certain class of boundary conditions studied in \cite{NeNoSil} that may be schematically written as
\[
\ \mbox{no-slip}\ \vu|_{\partial \Omega} = 0 \ +\   \mbox{natural boundary conditions of Neumann type.}
\]
Then the proof is basically the same as for Brenner's model.

\subsection{Numerical schemes}

Theorem \ref{TT1} may be useful in the study of convergence to certain dissipative numerical schemes for the barotropic Navier-Stokes system, meaning schemes preserving some form of the energy inequality. Such a scheme was proposed by Karlsen and Karper \cite{KarKar1}, and a rigorous proof of convergence to weak solutions finally was finally established by Karper \cite{Karp}. Karper's result applies to a certain class of pressures, notably
\[
p(\vr) = a \vr^\gamma \ \mbox{for}\ \gamma > 3 ,\ N =3.
\]
On the other hand, however, the consistency estimates cover a larger set for $\gamma > 3/2$, see Gallou{\"e}t et al.\ \cite{GalHerMalNov}. It can be shown that the consistency estimate
imply that the family of numerical solutions generate a (dissipative) measure-valued solution. In accordance with the conclusion of Theorem \ref{TT1}, the numerical solutions
will converge to a classical exact solution as soon as the latter exists. In fact this has been shown in \cite{FeHoMaNo} by means of a discrete analogue of the relative energy
inequality.

\section{Measure valued solutions with bounded density field}

We conclude our discussion by a simple example that indicates that the measure-valued solutions may be indeed an artifact of the theory as long as they emanate from sufficiently
regular initial data. The following result is a direct consequence of Theorem \ref{TT1} and a regularity criterion proved by
Sun, Wang, and Zhang \cite{SuWaZh1} stating that solutions of the barotropic Navier-Stokes system starting from smooth initial data remain smooth as long as their density component remains bounded. Since Theorem \ref{TT1} requires slightly better regularity than \cite{SuWaZh1}, we restrict ourselves to very regular initial data for which the necessary local existence result was proved in \cite[Proposition 2.1]{FeHoMaNo}.


\begin{Theorem} \label{TT2}

In addition to the hypotheses of Theorem \ref{TT1}, suppose that $\mu > 0$, $\eta = 0$, and
$\{\nu_{t,x},\mathcal{D}\}$ is a dissipative measure-valued solution to the barotropic Navier-Stokes system in $(0,T) \times \Omega$ emanating from smooth data, specifically,
\[
\nu_{0,x} = \delta_{[r_0(x), \vc{U}_0(x)]} \ \mbox{for a.a.}\ x  \in \Omega,
\]
where
\[
r_0 \in C^3(\Ov{\Omega}),\ r_0 > 0, \ \vc{U}_0 \in C^3(\Ov{\Omega}),\ \vc{U}_0|_{\partial \Omega} = 0, \ \Grad p(r_0) = \Div \mathbb{S} (\Grad \vc{U}_0).
\]
Suppose that the measure valued solution $\nu_{t,x}$ has bounded density component, meaning the support of the measure $\nu_{t,x}$ is confined to a strip
\[
0 \leq s \leq \Ov{\vr}\ \mbox{for a.a}\ (t,x) \in (0,T) \times \Omega.
\]

Then $\mathcal{D}=0$ and $\nu_{t,x}= \delta_{[ r(\tau,x), \vc{U}(\tau,x) ]} \ \mbox{for a.a.}\ \tau \in (0,T),\ x \in \Omega$, where
$[r, \vc{U}]$ is  a classical smooth solution of the barotropic Navier-Stokes system in $(0,T) \times \Omega$.

\end{Theorem}

\begin{Remark}
Note that  $\Grad p(r_0) = \Div \mathbb{S} (\Grad \vc{U}_0)$ is the standard compatibility condition associated to~\eqref{I3}.
\end{Remark}

\color{black}

\bProof

As stated in \cite[Proposition 2.1]{FeHoMaNo}, the compressible Navier-Stokes system (\ref{I1}--\ref{I3}) endowed with the initial data $[r_0, \vc{U}_0]$ admits a local in time
classical solutions fitting the regularity class required in Theorem \ref{TT1}. Thus the measure-valued solution coincides with the classical one on its life span. However, as the density component is bounded, the result of Sun, Wang, and Zhang \cite{SuWaZh1} asserts that the classical solution can be extended up to the time $T$.

\qed

\begin{Remark} \label{rrem}

The assumption that the bulk viscosity $\eta$ vanishes is a technical hypothesis used in~\cite{SuWaZh1}.

\end{Remark}

\def\cprime{$'$} \def\ocirc#1{\ifmmode\setbox0=\hbox{$#1$}\dimen0=\ht0
  \advance\dimen0 by1pt\rlap{\hbox to\wd0{\hss\raise\dimen0
  \hbox{\hskip.2em$\scriptscriptstyle\circ$}\hss}}#1\else {\accent"17 #1}\fi}


\end{document}